\documentclass[12pt,leqno]{amsart}

\usepackage{color}
\usepackage{enumerate}
\usepackage{graphicx}
\usepackage{epsfig}
\usepackage{amssymb,amsmath,amsthm,amsfonts}

\usepackage[letterpaper, margin=1in]{geometry} 

\usepackage {latexsym}

\usepackage{bbm}

\allowdisplaybreaks

\newcommand{\nc}{\newcommand}
\nc{\les}{\lesssim}
\nc{\nit}{\noindent}
\nc{\nn}{\nonumber}
\nc{\D}{\partial}
\nc{\diff}[2]{\frac{d #1}{d #2}}
\nc{\diffn}[3]{\frac{d^{#3} #1}{d {#2}^{#3}}}
\nc{\pdiff}[2]{\frac{\partial #1}{\partial #2}}
\nc{\pdiffn}[3]{\frac{\partial^{#3} #1}{\partial{#2}^{#3}}}
\nc{\abs}[1] {\lvert #1 \rvert}
\nc{\cAc}{{\cal A}_c}
\nc{\cE}{{\cal E}}
\nc{\cF}{{\cal F}}
\nc{\cP}{{\cal P}}
\nc{\cV}{{\cal V}}
\nc{\cQ}{{\cal Q}}
\nc{\cGin}{{\cal G}_{\rm in}}
\nc{\cGout}{{\cal G}_{\rm out}}
\nc{\cO}{{\cal O}}
\nc{\Lav}{{\cal L}_{\rm av}}
\nc{\cL}{{\cal L}}
\nc{\cB}{{\cal B}}
\nc{\cZ}{{\cal Z}}
\nc{\cR}{{\cal R}}
\nc{\cT}{{\cal T}}
\nc{\cY}{{\cal Y}}
\nc{\cX}{{\cal X}}
\nc{\cXT}{{{\cal X}(T)}}
\nc{\cBT}{{{\cal B}(T)}}
\nc{\vD}{{\vec \mathcal{D}}}
\nc{\efield}{\mathcal{E}}
\nc{\mE}{\mathcal{E}}
\nc{\vE}{{\vec \efield}}
\nc{\vB}{{\vec \mathcal{B}}}
\nc{\vH}{{\vec \mathcal{H}}}
\nc{\F}{  \mathcal{F} }
\nc{\ty}{{\tilde y}}
\nc{\tu}{{\tilde u}}
\nc{\tV}{{\tilde V}}
\nc{\Pc}{{\bf P_c}}
\nc{\bx}{{\bf x}}
\nc{\bX}{{\bf X}}
\nc{\bXYZ}{{\bf XYZ}}
\nc{\bY}{{\bf Y}}
\nc{\bF}{{\bf F}}
\nc{\bS}{{\bf S}}
\nc{\dV}{{\delta V}}
\nc{\dE}{{\delta E}}
\nc{\TT}{{\Theta}}
\nc{\dPsi}{{\delta\Psi}}
\nc{\order}{{\cal O}}
\nc{\Rout}{R_{\rm out}}
\nc{\eplus}{e_+}
\nc{\eminus}{e_-}
\nc{\epm}{e_\pm}
\nc{\eps}{\varepsilon}
\nc{\vnabla}{{\vec\nabla}}
\nc{\G}{\Gamma}
\nc{\w}{\omega}
\nc{\mh}{h}
\nc{\mg}{g}
\nc{\vphi}{\varphi}
\nc{\tlambda}{\tilde\lambda}
\nc{\be}{\begin{equation}}
\nc{\ee}{\end{equation}}
\nc{\ba}{\begin{eqnarray}}
\nc{\ea}{\end{eqnarray}}

\nc{\g}{\gamma}
\nc{\ol}{\overline}

\newtheorem{theorem}{Theorem}[section]
\newtheorem{lemma}[theorem]{Lemma}
\newtheorem{prop}[theorem]{Proposition}

\newtheorem{defin}[theorem]{Definition}
\newtheorem{rmk}[theorem]{Remark}

\nc{\pT}{\partial_T}
\nc{\pz}{\partial_z}
\nc{\pt}{\partial_t}
\nc{\la}{\langle}
\nc{\ra}{\rangle}
\nc{\infint}{\int_{-\infty}^{\infty}}
\nc{\halfwidth}{6.5cm}
\nc{\figwidth}{10cm}
\newcommand{\f}{\frac}

\nc{\nlayers}{L} \nc{\nsectors}{M}
\nc{\indicator}{\mathbf{1}}
\nc{\Rhole}{R_{\rm hole}}
\nc{\Rring}{R_{\rm ring}}
\nc{\neff}{n_{\rm eff}}
\nc{\Frem}{F_{\rm rem}}
\nc{\R}{\mathbb R}
\nc{\C}{\mathbb C}
\nc{\Z}{\mathbb Z}
\nc{\DD}{\Delta}
\nc{\cD}{\mathcal D}
\nc{\lnorm}{\left\|}
\nc{\rnorm}{\right\|}
\nc{\rnormp}{\right\|_{\ell^{p,\eps}}}
\nc{\rar}{\rightarrow}
\nc{\mR}{\mathcal R}
\nc{\oo}{\"o}   

\nc{\os}{\overset{o}}
\begin{document}

\begin{abstract}

	We consider the fourth order Schr\"odinger operator $H=\Delta^2+V$ and show that if there are no eigenvalues or resonances in the absolutely continuous spectrum of $H$ that the solution operator $e^{-itH}$ satisfies a large time integrable $|t|^{-\f54}$ decay rate between weighted spaces.  This bound improves what is possible for the free case in two directions; both better time decay and smaller spatial weights.  In the case of a mild resonance at zero energy, we derive the operator-valued expansion $e^{-itH}P_{ac}(H)=t^{-\f34} A_0+t^{-\f54}A_1$ where $A_0:L^1\to L^\infty$ is an operator of rank at most four and $A_1$ maps between polynomially weighted spaces.  

\end{abstract}

\title[Weighted dispersive estimates for fourth order Schr\"odinger]{Time integrable weighted dispersive estimates for the fourth order Schr\"odinger equation in three dimensions}

\author[Goldberg, Green]{Michael Goldberg and William~R. Green}
\thanks{The  first  author  is  supported  by  Simons  Foundation  Grant  635369. The second author is supported by  Simons  Foundation  Grant  511825. }
\address{Department of Mathematics\\
	University of Cincinnati \\
	Cincinnati, OH 45221 U.S.A.}
\email{goldbeml@ucmail.uc.edu}
 \address{Department of Mathematics\\
Rose-Hulman Institute of Technology \\
Terre Haute, IN 47803, U.S.A.}
\email{green@rose-hulman.edu}

%\subjclass{35Q41, 42B20}

\maketitle

\section{Introduction}

In this paper we study the dynamics of the solution operator to the fourth order Schr\"odinger equation
\be \label{eq:4th order}
i\psi_t=(\Delta^2+V)\psi, \quad \psi(0,x)=f(x), \qquad x\in \R^3.
\ee 
Here $V$ is a real-valued decaying potential.
For convenience, we denote the fourth order Schr\"odinger operator $H:=\Delta^2+V$, and $P_{ac}(H)$ is projection onto the absolutely continuous subspace of $L^2(\R^3)$ associated to the self-adjoint operator $H$.

In the free case, when $V=0$, it is well known (c.f. \cite{BKS}) that the solution operator satsfies the dispersive estimate
$$
\|e^{-it\Delta^2} f\|_{L^\infty (\R^3)}\les |t|^{-\f34}\|f\|_{L^1 (\R^3)}.
$$
The notation $A\les B$ means there exists an absolute constant $C$ so that $|A|\leq C|B|$.  In \cite{egt}, it was shown that the solution operator $e^{-itH}P_{ac}(H)$ for \eqref{eq:4th order} satisfies the same dispersive bound as above for a generic class of short range perturbations.  In this paper we show that a generic perturbation can lead to faster time decay.  
\begin{theorem}\label{thm:reg}
	If zero is a regular point of the spectrum of $H$, and there are no embedded eigenvalues in the spectrum of $H$, then we have the following
	$$
	\|e^{-itH}P_{ac}(H)f\|_{L^{\infty,-1} (\R^3)}\les |t|^{-\f54} \|f\|_{L^{1,1} (\R^3)}, \qquad |t|\geq 1
	$$
	provided $|V(x)|\les \la x\ra^{-\beta}$ for some $\beta>7$.
	
\end{theorem}

Here $\la x\ra:=(1+|x|^2)^{\f12}$. The polynomially weighted $L^p$ spaces are defined by
$$
L^{p,\sigma}(\R^3)=\{ f\,:\, \la \cdot \ra^\sigma f\in L^p(\R^3) \}, \qquad 1\leq p\leq \infty, \quad \sigma \in \R.
$$
This result shows that that one can obtain faster, time-integrable decay if the perturbation $V$ removes the natural zero energy resonance, at the cost of spatial weights.  One can interpolate between the result in Theorem~\ref{thm:reg} and the result in \cite{egt} to show the bound
\be \label{eq:interp bd}
\|e^{-itH}P_{ac}(H)f\|_{L^{\infty,-\sigma} (\R^3)}\les |t|^{-\f34-\f\sigma2} \|f\|_{L^{1,\sigma} (\R^3)}, \qquad |t|\geq 1, \quad 0\leq \sigma \leq 1.
\ee
The condition that zero energy is regular prohibits the existence of solutions to $H\psi=0$ with $\psi \in L^\infty (\R^3)$.  If there is a non-trivial solution with $\psi\in L^\infty(\R^3)$ we say there is a zero energy resonance (or eigenvalue if $\psi \in L^2(\R^3)$), see \cite{egt} for a detailed study of these threshold phenomena.  We say $\psi$ is a resonance of the first kind if $\psi \in L^\infty (\R^3)$ but $\psi \notin L^p(\R^3)$ for any $p<\infty$.  %We note that the free operator $H_0 = \Delta^2$ is excluded from Theorem~\ref{thm:reg} due to the existence of the  zero energy resonance of the first kind, the constant function $\psi(x)=1$.  
%To quantify the resonance phenomenon more precisely, we prove  

The free operator $H_0 = \Delta^2$ is excluded from Theorem~\ref{thm:reg} due to the existence of the  zero energy resonance of the first kind, namely the constant function $\psi(x)=1$.  Consequently, it does not obtain the same decay rate.
There is an explicit solution formula in the free case
$$
e^{-it\Delta^2} f(x) = t^{-\frac{3}{4}} \int_{\R^3} \mathcal K\big(t^{-\frac{1}{4}}(x-y)\big)f(y)\,dy ,
$$
with the kernel $\mathcal K(x)$ being the inverse Fourier transform of $e^{-i|\xi|^4}$.  Since $\mathcal K(x)$ is bounded, smooth, and radially symmetric, it satisfies $|\mathcal K(x) - \mathcal K(0)| \les |x|^2$, which in turn makes 
$|\mathcal K(t^{-1/4}(x-y)) - \mathcal K(0)| \les t^{-\frac{1}{2}}|x-y|^2 \les
t^{-\frac{1}{2}}\la x\ra^2 \la y \ra^2$.  If we define $P_1$ to be the rank-one projection onto the constant function $\psi$, it follows that
$$
\|e^{-it\Delta^2} f - t^{-\frac{3}{4}}\mathcal K(0) P_1 f\|_{L^{\infty,-2} (\R^3)}\les |t|^{-\f54}\|f\|_{L^{1,2} (\R^3)}.
$$
Remark~\ref{rmk:ests} below provides an alternate derivation of this bound for the free evolution.  In Theorem~\ref{thm:swave} we show that this behavior is representative of what occurs when there is a resonance of the first kind.  Comparing the two statements, we see that when a perturbation removes the resonance at zero, the result of Theorem~\ref{thm:reg} improves upon the free case in two ways: both the overall time decay and the required spatial weights in the $|t|^{-\f54}$ part of the solution.

\begin{theorem}\label{thm:swave}
	If there is a resonance of the first kind at zero, and there are no embedded eigenvalues in the spectrum of $H$,  then we have the following
	$$
	\|(e^{-itH}P_{ac}(H)-F_{t})f\|_{L^{\infty,-2} (\R^3)}\les |t|^{-\f54} \|f\|_{L^{1,2} (\R^3)}, \qquad |t|\geq 1
	$$
	with $F_t$ a time-dependent operator of rank at most four that satisfies $\|F_t\|_{1\to \infty} \les |t|^{-\f34}$.
	Provided $|V(x)|\les \la x\ra^{-\beta}$ for some $\beta>11$.
	
\end{theorem}

Effectively, the resonant perturbed evolution has two pieces.  The time-decay of $F_t$ matches the natural decay rate of the free equation, and it has rank at most four, the maximum dimension of the zero energy resonance space if there is only a resonance of the first kind.  We explicitly construct this operator in \eqref{Ft defn} below. % See Remark~7.9 in \cite{egt}.  
The remaining piece enjoys the same rate of time-decay as when zero energy is regular.    In all cases the small time estimates are  
$$
\|e^{-itH}P_{ac}(H)f\|_{L^{\infty} (\R^3)}\les |t|^{-\f34} \|f\|_{L^{1} (\R^3)}, \qquad |t|<1,
$$
due to the high energy portion of the evolution.
Hence the dispersive bounds we prove are integrable in time over the entire real line.  If integrability in time is the main concern, our proof of Theorem~\ref{thm:swave} is easily modified to show the bound
$$
\|(e^{-itH}P_{ac}(H)-F_{t})f\|_{L^{\infty,-2} (\R^3)}\les |t|^{-1-} \|f\|_{L^{1,2} (\R^3)}, \qquad |t|\geq 1,
$$
which necessitates only $|V(x)|\les \la x\ra^{-\beta}$ for some $\beta>9$, see Remark~\ref{rmk:wtd} below.

The fourth order operators $H=\Delta^2-\epsilon \Delta +V$ with $\epsilon\in \{0,\pm1\}$ were formulated to study the propagation of laser beams in a bulk medium with Kerr non-linearity, \cite{K,KS}.  Dispersive estimates were considered in \cite{BN,BKS}.  Much of the analysis of the dispersive estimates has been done for the case of $\epsilon=0$, with recent work spurred on by the the weighted $L^2$ based analysis of Feng, Soffer and Yao in dimensions $n\geq 5$, \cite{fsy}.  The effect of zero energy resonances was then investigated by Toprak and the second author, \cite{GT4} in four dimensions and Erdo\smash{\u{g}}an, Toprak and the second author in three dimensions in the $L^1\to L^\infty$ setting, \cite{egt}.  Weighted $L^2$ estimates are considered in \cite{FWY} in higher dimensions.  The case of $\epsilon=1$ is considered in \cite{feng}, the case of $\epsilon =-1$ is open.  Higher order Schr\"odinger operators $(-\Delta)^m+V$ are studied in \cite{soffernew}.  In contrast to the second order Schr\"odinger operator, for higher order operators there can be embedded eigenvalues even for perturbations by compactly supported, smooth potentials. Examples can be constructed using the fact that the analogous Helmholtz equation $(\Delta^2 - \lambda^4)\psi = 0$ admits exponentially decaying solutions in $\R^3$. Theorem 1.11 and Remark 1.12 in \cite{soffernew} establish a criteria to rule out embedded eigenvalues. Of particular relevance to the class of potentials we consider is the criterion that a repulsive potential, one satisfying  $V (\gamma x) \leq  V (x)$ for all $\gamma > 1$,  does not have any embedded eigenvalues.

The recent study of the dispersive bounds for fourth order equations builds heavily on results for the (second order) Schr\"odinger operator, see \cite{JenKat,JSS,Mur,GS,Sc2,ES,goldE,eg2,EGG,GG1,GG2} 
for study of the effect of threshold obstructions on the dispersive estimates.  Generically a threshold obstruction leads to a slower large time-decay in the dispersive estimate.  The result of Theorem~\ref{thm:reg} mirrors the result of Erdo\smash{\u{g}}an and the second author in \cite{eg3} for the two dimensional Schr\"odinger operator, where one can obtain an integrable time-decay at the cost of spatial weights.  The result of Theorem~\ref{thm:swave} mirrors the result of Toprak, \cite{Top}, which connected results in \cite{eg2,eg3} to quantify the effect of a `mild' resonance on the evolution of the two dimensional Schr\"odinger operator.  Such estimates are of use in the study of asymptotic stablity of special solutions to nonlinear equations.  Nonlinear fourth order equations have been studied, see for example \cite{Lev,Lev1,LS,CLB,P,P1,CLB1,MXZ1,MXZ2,DDY,Dinh,FWY1}.

As is the standard set-up for proving dispersive estimates, we utilize the spectral theory of $H$ to reduce to bounding oscillatory integrals in the spectral parameter.  First, we relate the resolvent of the free fourth order operator to the free second order Schr\"odinger operator (c.f. \cite{fsy} or apply the second resolvent identity) to see that
\begin{align}\label{RH_0 rep}
R (H_0; z):=( \Delta^2 - z)^{-1} = \frac{1}{2\sqrt{z}} \Big( (-\Delta-\sqrt{z})^{-1} - (-\Delta- \sqrt{-z})^{-1} \Big), \quad z\in \mathbb C\setminus[0,\infty).
\end{align}
For convenience, we denote $R_0(z)=(-\Delta-z)^{-1}$ for the usual Schr\"odinger resolvent.  For the potentials considered we have a Weyl criterion and hence $\sigma_{ac}(H)=\sigma_{ac}(\Delta^2)=[0,\infty)$.
Let $\lambda \in \R^{+}$, we define the limiting resolvent operators by
\begin{align}
&\label{RH_0 def}R^{\pm}(H_0; \lambda ) := R (H_0; \lambda \pm i0 )= \lim_{\epsilon \to 0^+}(\Delta^2 - ( \lambda  \pm i\epsilon))^{-1}, \\
& \label{Rv_0 def} R_V^{\pm}(\lambda ) := R_V (\lambda  \pm i0 )= \lim_{\epsilon \to 0^+}(H - ( \lambda \pm i\epsilon))^{-1}.
\end{align}

Note that using the representation \eqref{RH_0 rep} for $R (H_0;z)$ in   definition \eqref{RH_0 def} with $z=w^4$ for $w$ in the first quandrant of the complex plane, and taking limits as $w$ approaches the real ($w\to \lambda$) and imaginary axes ($w\to i\lambda$) from the first quadrant, we obtain 
\be\label{eq:4th resolvdef}
R^{\pm}(H_0; \lambda^4)= \frac{1}{2 \lambda^2} \Big( R^{\pm}_0(\lambda^2) - R_0(-\lambda^2) \Big),\,\,\,\lambda>0.
\ee 
Note that $R_0(-\lambda^2 )$ is a bounded operator on $L^2$ when $\lambda>0$. Further, by Agmon's limiting absorption principle, \cite{agmon}, $R^{\pm}_0(\lambda^2)$ is well-defined between weighted $L^2$ spaces. Therefore, $R^{\pm}(H_0; \lambda^4)$ is also well-defined between these weighted spaces. This property extends to $R_V^{\pm}(\lambda )$ provided there are no eigenvalues in the spectrum, \cite{fsy}.  While the absence of embedded eigenvalues is a standard assumption in the analysis of dispersive estimates, we note that \cite{soffernew} establishes classes of potentials for which embedded eigenvalues cannot exist but we choose to keep the absence of embedded eigenvalues as an overarching assumption.
Using the functional calculus and Stone's formula we write 
\begin{align}
\label{stone}
\ e^{-itH}    P_{ac}(H) f(x) = \frac1{2\pi i}  \int_0^{\infty} e^{-it\lambda}   [R_V^+(\lambda)-R_V^{-}(\lambda)]  f(x) d\lambda.
\end{align}
Here the difference of the perturbed resolvents provides the spectral measure.   We make the change of variables $\lambda \mapsto \lambda^4$ to bound
$$
\frac2{\pi i}  \int_0^{\infty} e^{-it\lambda^4} \lambda^3  [R_V^+(\lambda^4)-R_V^{-}(\lambda^4)]  f(x) d\lambda.
$$

For the convenience of the reader, we have gathered the notation and terminology we use throughout the paper.  Similar notation is used in previous study of the fourth order Schr\"odinger operator in \cite{GT4,egt}.
For an operator $\mE(\lambda) $, we write $\mE(\lambda)=O_k(\lambda^{-\alpha})$ if its kernel $\mE(\lambda)(x,y)$   has the property
\be\label{O1lambda} \sup_{x,y\in\R^3,  			\lambda>0}\big[\lambda^{\alpha}|\mE(\lambda)(x,y)|+\lambda^{\alpha+1} |\partial_\lambda\mE(\lambda)(x,y)|+\dots +\lambda^{\alpha+k}|\partial_\lambda^k\mE(\lambda)(x,y)|\big]<\infty.
\ee 
Similarly, we use the notation $\mE(\lambda)=O_k(\lambda^{-\alpha}g(x,y))$ if $\mE(\lambda)(x,y)$   satisfies
\be\label{O1lambdag}   |\mE(\lambda)(x,y)|+\lambda  |\partial_\lambda\mE(\lambda)(x,y)|+\dots +\lambda^{k}|\partial_\lambda^k\mE(\lambda)(x,y)|\les \lambda^{-\alpha}g(x,y).
\ee
For operators on $L^2(\R^3)$ we also recall the following terminology from \cite{Sc2,eg2}: An operator $T:L^2(\R^3) \to   L^2(\R^3)$ with kernel $T(\cdot,\cdot)$ is absolutely bounded if the operator with kernel $|T(\cdot,\cdot)|$ is bounded from $  L^2(\R^3)$ to $ L^2(\R^3)$ as well. 	 
%\begin{defin}
%	We say an operator $T:L^2(\R^3) \to   L^2(\R^3)$ with kernel
%	$T(\cdot,\cdot)$ is absolutely bounded if the operator with kernel
%	$|T(\cdot,\cdot)|$ is bounded from $  L^2(\R^3)$ to $ L^2(\R^3)$. 	
%\end{defin}
Finite rank and Hilbert-Schmidt operators are immediately absolutely bounded.
The symbol $\Gamma_\theta$  denotes a $\lambda$ dependent absolutely bounded operator  satisfying 
\be\label{eq:Gamma def}
\big \||\Gamma_\theta|\big \|_{L^2\to L^2}+ \lambda \big \||\partial_\lambda \Gamma_\theta|\big \|_{L^2\to L^2}+ \lambda^2 \big \||\partial_\lambda^2 \Gamma_\theta|\big \|_{L^2\to L^2}  \les \lambda^{\theta},\quad \lambda>0.
\ee
In other words it is similar to $O_2(\lambda^\theta)$ except in the family of bounded operators on $L^2$ instead of as maps from $L^1$ to $L^\infty$.  The operator represented by $\Gamma_\theta$ may be different in each occurrence.   

We use the smooth, even low energy cut-off $\chi$ defined by $\chi(\lambda)=1$ if $|\lambda|<\lambda_0\ll 1$ and $\chi(\lambda)=0$ when $|\lambda|>2\lambda_0$ for some sufficiently small constant $0<\lambda_0\ll1$.  In analyzing the high energy we utilize the complementary cut-off $\widetilde \chi(\lambda)=1-\chi(\lambda)$. Throughout the paper, an exponent $a-$ denotes $a-\epsilon$ for an arbitrarily small, but fixed $\epsilon>0$.  Similarly, $a+$ indicates $a+\epsilon$.

The paper is organized as follows.  We begin in Section~\ref{sec:free} by showing that the unperturbed solution operator satisfies the low energy bounds in Theorem~\ref{thm:swave}.  In Section~\ref{sec:regular} we prove the low energy portion of the evolution satisfies Theorem~\ref{thm:reg} and show that perturbing the free operator by a potential that removes the natural zero energy resonance leads to faster time decay at the cost of spatial weights.  In Section~\ref{sec:swave} we establish the low energy portion of Theorem~\ref{thm:swave} and construct the operator $F_t$.  Finally, in Section~\ref{sec:high} we control the high energy portion of the evolution completing the proofs of Theorems~\ref{thm:reg} and \ref{thm:swave}.

\section{The Free  Evolution}\label{sec:free}

In this section we obtain expansions for the free fourth order Schr\"odinger resolvent operators $R^{\pm}(H_0; \lambda^4)$, using the identity \eqref{RH_0 rep}  and the Bessel function representation of the Schr\"odinger free resolvents $R^{\pm}_0(\lambda^2)$. We use these expansions to establish  dispersive estimates for the free fourth order Schr\"odinger evolution, and throughout the remainder of the paper to study the spectral measure for the perturbed operator.  Our expansions are similar to those obtained in \cite{egt}, but our analysis requires control on further derivatives in all cases.

Recall the expression of the free Schr\"odinger resolvents in dimension three, (see \cite{GS} for example) 
$$
R^{\pm}_0(\lambda^2) (x,y)= \frac{e^{\pm i \lambda|x-y|}}{4 \pi |x-y|} .
$$
Therefore, by  \eqref{eq:4th resolvdef}, 
\be\label{eq:R0lambda}
R^\pm (H_0, \lambda^4) (x,y)= \frac{1}{2 \lambda^2}  \Bigg( \frac{e^{\pm i \lambda|x-y|}}{4 \pi |x-y|}  -  \frac{e^{-\lambda|x-y|}}{4 \pi |x-y|}  \Bigg).
\ee
We have the following representation for the $R(H_0, \lambda^4)$.
\begin{align} \label{eq:R0low}
R^\pm(H_0, \lambda^4) (x,y) = \frac{a^{\pm}}{\lambda}+\mathcal E_0(\lambda), \qquad   \mathcal E_0(\lambda)= G_0 +\mathcal E_1(\lambda), \qquad \mE_1(\lambda)= a_1^{\pm}  \lambda G_1 +\mathcal E_2(\lambda).
%+ a_3^{\pm} \lambda^3 G_3+ \lambda^4 G_4 +  O_2(\lambda^5 |x-y|^6) .
\end{align} 
Here 
\begin{align}\label{adef}
&a^{\pm}:= \frac{1\pm i}{8 \pi}, \quad a_1^{\pm}= \frac{1\mp i }{ 8 \pi \cdot (3!)}, \quad G_0 (x,y) = -\frac{|x-y|}{8\pi}, \quad G_1 (x,y) = |x-y|^2.% \quad a_3^{\pm} =\frac {1\pm i }{8 \pi \cdot (5!)}, \quad G_0 (x,y) = - \frac{|x-y|}{8 \pi}, \\\label{Gdef}
\end{align}
Notice that   $G_0=(\Delta^2)^{-1}$. 
When $\lambda|x-y|\les 1$, we have
$\mathcal E_0(\lambda)=O_2(\lambda^{-1+\gamma} |x-y|^{ \gamma} )$, $\mathcal E_1(\lambda)=O_2(\lambda^\gamma |x-y|^{1+\gamma})$, $ 0 \leq \gamma \leq 1$ and $ \mathcal E_2(\lambda)=O_2(\lambda^{1+\gamma} |x-y|^{2+\gamma} ), $ for $ 0 \leq \gamma \leq 2$.  The larger range of $\gamma$ for $\mathcal E_2(\lambda)$ follows since the $\lambda^2$ terms in the series expansion of \eqref{eq:R0lambda} cancel.
When $\lambda |x-y|\gtrsim 1$, the expansion remains valid for the first derivative.

For control of the second derivative, we note that if $\lambda |x-y|\gtrsim 1$,
$$
|\partial_{\lambda}^2R^\pm(H_0, \lambda^4) (x,y)|\les \frac{|x-y|}{\lambda^2}(\lambda |x-y|)^\ell, \qquad \text{for any }\ell \geq 0.
$$
Hence, $\partial_\lambda^2 \mathcal E_j(\lambda)$, $j=0,1$ satisfies the same bounds when $\lambda |x-y|\gtrsim 1$.
Thus, we may write
\begin{align}
\mathcal E_0(\lambda)=O_1(\lambda^{-1+\gamma}|x-y|^\gamma ),  \qquad
\mathcal E_0(\lambda)=O_2(|x-y| ),\quad 0\leq \gamma \leq 1, \label{eq:E0 bd} \\
\mathcal E_1(\lambda)=O_2(\lambda^{\gamma}|x-y|^{1+\gamma} ),\quad 0\leq \gamma \leq 1,\label{eq:E1 bd}\\
\mathcal E_2(\lambda)=O_2(\lambda^{1+\gamma}|x-y|^{2+\gamma}), \quad 0\leq \gamma \leq 2.\label{eq:E2 bd}
\end{align}

The following lemma is used repeatedly to obtain low energy dispersive estimates. 
\begin{lemma}\label{lem:t+a4bound} Fix $0<\alpha<4$. Assume that $\mE(\lambda)=O_2(\lambda^{ \alpha})$ for $0<\lambda\les 1$, 
	then we have the bound
	\be\label{eq:t+a4bound}
	\bigg|
	\int_{0}^\infty e^{-it\lambda^4}  \chi(\lambda) \lambda^3  \mE(\lambda)\, d\lambda
	\bigg| \les \la t\ra^{-1-\f\alpha4}.
	\ee
\end{lemma}

\begin{proof}
	By the support condition  and since $0<\alpha$,  the integral is bounded.  Now, for $|t|>1$ we rewrite the integral in \eqref{eq:t+a4bound} as
	$$
	\int_{0}^{t^{-\f14}} e^{-it\lambda^4} \lambda^3  \chi(\lambda)\mathcal E(\lambda)\, d\lambda+\int_{t^{-\f14}}^\infty e^{-it\lambda^4} \lambda^3 \chi(\lambda)\mathcal E(\lambda)\, d\lambda:=I+II.
	$$
	We see that
	$$
	|I|\leq \int_0^{t^{-\f14}} \lambda^{3+\alpha}\, d\lambda \les t^{-1-\f\alpha4 }.
	$$
	For the second term, we use $\partial_\lambda e^{-it\lambda^4}/(-4it)=e^{-it\lambda^4} \lambda^3$ to  integrate by parts twice,
	\begin{align*}
	II 
	%=  \frac{e^{-it\lambda^4} \mathcal E(\lambda)}{4it}\bigg|_{t^{-\f14}} - \frac{1}{4it} \int_{t^{-\f14}}^\infty e^{-it\lambda^4}  \mathcal E'(\lambda) \, d\lambda   \\
	=  \frac{e^{-it\lambda^4} \mathcal E(\lambda)}{-4it}\bigg|_{t^{-\f14}} -  \frac{e^{-it\lambda^4} \mathcal E'(\lambda)}{(4it)^2 \lambda^3}\bigg|_{t^{-\f14}}
	+ \frac{1}{(4\pi it)^2}\int_{t^{-\f14}}^\infty e^{-it\lambda^4} \partial_{\lambda}\bigg( \frac{\mathcal E'(\lambda)}{\lambda^3} \bigg)   \, d\lambda .
	\end{align*}
	Then, we have
	$$
	|II| \les  \frac{ |\mathcal E(t^{-\f14})|}{t}+\frac{ |\mathcal E'(t^{-\f14})|}{t^{\f54}}+\frac{1}{t^2}\int_{t^{-\f14}}^\infty \lambda^{\alpha-5}\, d\lambda \les t^{-1-\f{\alpha}{4}}.
	$$

\end{proof}

Another useful bound is 
\begin{lemma}\label{lem:t-34bound} \cite[Lemma~3.1]{egt} Fix $0<\alpha<4$. Assume that $\mE(\lambda)=O_1(\lambda^{-\alpha})$ for $0<\lambda\les 1$, 
	then we have the bound
	\be\label{eq:t-34bound}
	\bigg|
	\int_{0}^\infty e^{-it\lambda^4}  \chi(\lambda) \lambda^3  \mE(\lambda)\, d\lambda
	\bigg| \les \la t\ra^{-1+\f\alpha4}.
	\ee
\end{lemma}

These can be combined to give a characterization of the free evolution.
\begin{lemma}\label{lem:free bound}
	For any $-1\leq\sigma\leq 1$
	We have the expansion
	\begin{multline*}
	\int_{0}^\infty e^{-it\lambda^4} \chi(\lambda) \lambda^3  [R^+(H_0,\lambda^4)-R^-(H_0,\lambda^4)](x,y)\, d\lambda\\
	= (a^+-a^-)\int_0^\infty e^{-it\lambda^4} \chi(\lambda) \lambda^2 \, d\lambda + O(\la t\ra ^{-1 -\f{\sigma}{4}} \la x\ra^{1+\sigma} \la y \ra^{1+\sigma}  )  .	
	\end{multline*}

\end{lemma}

\begin{proof}
	Using \eqref{eq:R0low} we have
	$$
	[R^+(H_0,\lambda^4)-R^-(H_0,\lambda^4)](x,y)=\frac{a^+-a^-}{\lambda}+ [\mathcal E_0^+-\mathcal E_0^-](\lambda).
	$$
	Using \eqref{eq:E0 bd}, $\mathcal E_0^+-\mathcal E_0^-=O_1(\lambda^{-1+\gamma} |x-y|^\gamma)$.
	Inserting this into the Stone's formula along with Lemma~\ref{lem:t-34bound}, noting $|x-y| \les \la x\ra \la y\ra$,  yields the claim. suffices to prove the desired bound for $-1+\gamma=\sigma<0$.  If $\sigma> 0$, we use \eqref{eq:E1 bd} and \eqref{eq:R0low} to write $\mathcal E_0^+ -\mathcal E_0^-=\mathcal E_1^+ -\mathcal E_1^-$ in Lemma~\ref{lem:t+a4bound}.  The remaining case of $\sigma=0$ follows by interpolating between the two proven cases of $\sigma=0+$ and $\sigma=0-$ respectively.
	
\end{proof}
\begin{rmk} \label{rmk:ests}  
	It was shown in \cite{BKS,egt} that the free operator satisfies the bound
	$$
	\| e^{i t \Delta^2}  \|_{ L^{1} \rightarrow L^{\infty}} \les t^{-\f34}. 
	$$ 
	A consequence of the bound above is the operator-valued expansion
	$$
	e^{i t \Delta^2}=|t|^{-\f34} A_0+ |t|^{-1-\f{\sigma}{4}} A_1
	$$
	where $A_0:L^1\to L^\infty$ is rank one and $A_1:L^{1,1+\sigma}\to L^{\infty, -1-\sigma}$ for any $-1\leq\sigma\leq 1$.
	
	Theorem~\ref{thm:swave} matches this expansion when $\sigma=1$, except possibly for the rank of $A_0$.
	Comparing this expansion to the result in Theorem~\ref{thm:reg}, we see that if zero energy is regular one obtains better time decay since $A_0\equiv 0$ in this case.  Further, the spatial weights required in Theorem~\ref{thm:reg} are one power smaller for the $|t|^{-\f54}$ term than for the free evolution.
\end{rmk}

\section{Weighted bounds when zero is regular} \label{sec:regular}

The main goal of this section is to establish the following 

\begin{prop}\label{prop:low reg}
	If zero is a regular point and $|V(x)|\les \la x\ra^{-\beta}$ for some $\beta>7$, then we have the following expansion for the perturbed resolvent in a sufficiently small neighborhood of $\lambda=0$:
	\be\label{eq:RV reg}
	R_V^\pm(\lambda^4)(x,y) = C_0+ O_2(\lambda\la x\ra \la y\ra).
	\ee
	
	In particular,
	\be \label{eq:reg Rv diff}
	R_V^+-R_V^- = O_2(\lambda \la x \ra \la y \ra).
	\ee
	
\end{prop}

The low energy portion of Theorem~\ref{thm:reg} follows from this bound and the oscillatory integral estimate in  Lemma~\ref{lem:t+a4bound}.

To understand \eqref{stone} for small energies, i.e. $ 0<\lambda \ll 1$, we use the symmetric resolvent identity.  We define $U(x)=$sign$(V(x))$, $v(x)=|V(x)|^{\f12}$, and write 
\begin{align} \label{resid}
R^{\pm}_V(\lambda^4)= R^{\pm}(H_0, \lambda^4) - R^{\pm}(H_0, \lambda^4)v (M^{\pm} (\lambda))^{-1} vR^{\pm}(H_0, \lambda^4) ,
\end{align}
where $M^{\pm}(\lambda) = U + v R^{\pm}(H_0, \lambda^4) v $. As a result, we need to obtain expansions for $(M^{\pm} (\lambda))^{-1}$.  The behavior of these operators as $\lambda \to 0$ depends on the type of resonances at zero energy, see Section~4 of \cite{egt}. We first develop an approriate expansion when zero is regular to construct the low energy spectral measure in Stone's formula, \eqref{stone}.

Let $T:= U + v  G_0 v$, and recall \eqref{eq:Gamma def}, we have the following expansions. 
\begin{lemma}\label{lem:M_exp} For  $0<\lambda<1$ define  $M^{\pm}(\lambda) = U + v R^{\pm}(H_0, \lambda^4) v $. Let $P=v\langle \cdot, v\rangle \|V\|_1^{-1}$ denote the orthogonal projection onto the span of $v$.  We have
	\begin{align}  
	\label{Mexp} M^{\pm}(\lambda)&= A^{\pm}(\lambda) +  M_0^\pm (\lambda),\\ 
	\label{Apm}
	A^{\pm}(\lambda)&= \frac{\|V\|_1 a^\pm}{\lambda} P+T,
	\end{align} 
	where $T:=U+vG_0v$ and
	$M_0^\pm (\lambda)=v \mathcal E_1(\lambda) v=\Gamma_{\gamma}$, for any $0\leq \gamma \leq  1$,  provided that $v(x)\les \la x \ra^{-\frac{5}{2}-\gamma-}$.  Where $\mathcal E_0(\lambda)$ is the operator from \eqref{eq:E0 bd}.
	Moreover,  
	\be\label{M_0}
	M_0^\pm(\lambda)=a_1^\pm \lambda vG_1 v+ \Gamma_{1+\gamma}, \qquad 0\leq \gamma \leq 2
	\ee
	provided that $v(x)\les \la x \ra^{-\frac{7}{2}-\gamma-}$. Here the operators and the error term are  Hilbert-Schmidt, and hence absolutely bounded operators. 
	
\end{lemma} 	

\begin{proof} 
	Using the expansion \eqref{eq:R0low} and the bounds on $\mathcal E_1(\lambda)$ and $\mathcal E_2(\lambda)$ in \eqref{eq:E1 bd} and \eqref{eq:E2 bd} respectively in the definition of $M^{\pm}(\lambda)$ and $M^{\pm}_0(\lambda)$ suffices to establish the bounds.
	
\end{proof}

The definition below classifies the type of resonances considered in this paper at the threshold energy. A more detailed study of the remaining types of resonances may be found in Definition~4.2 and Section~7 of \cite{egt}.
\begin{defin} \label{def:restype} 
	
	\begin{enumerate}[i)]
		
		\item  Let $Q:=I-P$. Zero is a regular point of the spectrum of $\Delta^2 +V$ provided $QTQ$ is invertible on $QL^2$. In that case we define $D_0:=(QTQ)^{-1}$ as an absolutely bounded operator on $QL^2$ provided $|V(x)|\les \la x\ra^{-5-}$ by Lemma~4.3 in \cite{egt}.  
		
		\item  Assume that zero is not regular point of the spectrum. Let $S_1$ be the Riesz projection onto the kernel of $QTQ$. Then $QTQ+S_1$ is invertible on $QL^2$. Accordingly,  we
		define $D_0 = (QT Q + S_1)^{-1}$,  as an operator on $QL^2$.  (This agrees with the previous definition since $S_1=0$  when zero is regular.) We say there is
		a resonance of the first kind at zero if the operator  $ T_1:=S_1TPTS_1- \frac{\|V\|_1}{3(8\pi)^2}  S_1vG_1vS_1$ is invertible
		on $S_1L^2$. 
		
	\end{enumerate}
\end{defin}

Note that $T$ is a compact perturbation of $U$.  Hence, the Fredholm alternative guarantees that $S_1$ is a finite-rank projection. Also, $  S_1 \leq Q $ and $Qv=0$, hence   $S_1v=0$. Second, since $T$ is a self-adjoint operator and $S_1$ is the Riesz projection onto its kernel, we have  $S_1 D_0= D_0 S_1 = S_1$.   

Recall from \eqref{Mexp} that $M^\pm(\lambda) =   A^\pm(\lambda) + M_0^\pm(\lambda)$. 
If zero is regular then we have the following expansion for $(A^\pm(\lambda))^{-1}$.
\begin{lemma}\label{regular}  \cite[Lemma~4.5]{egt} Let $ 0 < \lambda \ll 1$. If zero is regular point of the spectrum of $H$. Then, we have
	\begin{align} \label{M inverse regular}
	(A^\pm(\lambda))^{-1}= QD_0Q + g^{\pm}(\lambda)  S, 
	\end{align}
	where $g^{\pm}(\lambda)=(\frac{ a^\pm \|V\|_1}{\lambda} +c)^{-1} $ for some $c\in \R$, and
	\begin{align} \label{def:S}
	S =\left[\begin{array}{cc}
	P& -PTQD_0Q\\ -QD_0QTP &QD_0QTPTQD_0Q
	\end{array} \right],
	\end{align}
	is a self-adjoint, finite rank operator.
	Moreover, the same formula holds for $(A^\pm(\lambda)+S_1)^{-1}$ with $D_0=(Q(T+S_1)Q)^{-1}$ if   zero is not regular. 
\end{lemma}
This lemma and the definition preceeding it are of use in Section~\ref{sec:swave} when we consider a resonance of the first kind.  For the rest of this section, we only consider the case when zero is regular and $S_1=0$.  With the assumption that $ v(x) \les \la x\ra^{-\frac72-}$, we are able to conclude from Lemma~\ref{lem:M_exp} that $M_0^\pm(\lambda) = \Gamma_{1}$.

By Lemma~\ref{regular}, $A^+(\lambda)^{-1}=\Gamma_0$, and then
\begin{align}\label{eq:Minv exp reg}
M^+(\lambda)^{-1} &= A^+(\lambda)^{-1} - A^+(\lambda)^{-1}M_0^+(\lambda)(I + A^+(\lambda)^{-1}M_0)^{-1} A^+(\lambda)^{-1}\nn \\
&=A^+(\lambda)^{-1} - A^+(\lambda)^{-1}\Gamma_{1} A^+(\lambda)^{-1}.
\end{align}

Another application of Lemma~\ref{regular} leads
\begin{multline}\label{eq:RV exp reg}
R_V^+(\lambda^4) = R^+(H_0,\lambda^4) - R^+(H_0,\lambda^4)vA^+(\lambda)^{-1}vR^+(H_0,\lambda^4)\\ +R^+(H_0,\lambda^4)vA^+(\lambda)^{-1}\Gamma_{1}A^+(\lambda)^{-1}vR^+(H_0,\lambda^4)
\end{multline}
There is  significant cancellation in that the leading part of $R^+(H_0,\lambda^4)v$ is orthogonal to the range of $Q$.

\begin{lemma} \label{lem:RvQ}
	If $v(x) \les \la x\ra^{-\frac{5}{2}-}$, then $R^+(H_0,\lambda^4)vQ$ is $O_2(1)$ as an operator from $L^2(\R^3)$ to
	$L^{\infty, -1}(\R^3)$.
\end{lemma}
\begin{proof}
	For all $\psi \in QL^2$, $\int v(y)\psi(y)\,dy = 0$, hence $R^+(H_0,\lambda^4)vQ = (R^+(H_0,\lambda^4) - \frac{a^+}{\lambda})vQ$.  It follows from \eqref{eq:R0low} and \eqref{eq:E0 bd} that $(R^+(H_0,\lambda^4) - \frac{a^+}{\lambda}) = O_2(|x-y|)$,
	which is also $O_2(\la x\ra \la y \ra)$. The result follows as long as $\la y\ra v \in L^2(\R^3)$.
\end{proof}

We first consider the final term in \eqref{eq:RV exp reg}.
\begin{lemma} \label{lem:RvA}
	The kernel of $g^+(\lambda)R^+(H_0,\lambda^4)$ is $O_2(\la x\ra \la y\ra)$. As a consequence, if $v(x) \les \la x\ra^{-\frac{5}{2}-}$ then $R^+(H_0,\lambda^4)vA^+(\lambda)^{-1} = O_2(1)$ as an operator from $L^2(\R^3)$ to $L^{\infty,-1}(\R^3)$, and 
	$$
	R^+(H_0,\lambda^4)vA^+(\lambda)^{-1}\Gamma_1A^+(\lambda)^{-1}vR^+(H_0,\lambda^4) = O_2(\lambda\la x\ra\la y\ra).
	$$
\end{lemma}

\begin{proof}
	The scalar function $g^+(\lambda) =\lambda(a^+\|V\|_1 + c\lambda)^{-1}$ is both $O_2(\lambda)$ and $O_2(1)$.  Recalling \eqref{eq:E0 bd}, we see that $R^+(H_0,\lambda^4) = \frac{a^+}{\lambda} + O_2(|x-y|)$.  Then the products
	$\frac{a^+g^+(\lambda)}{\lambda} = O_2(1)$ and $O_2(|x-y|)g^+(\lambda) = O_2(|x-y|)$, and both of these are $O_2(\la x\ra \la y\ra)$.
	
	Recall that $R^+(H_0,\lambda^4)vA^+(\lambda)^{-1} = R^+(H_0,\lambda^4)v(QD_0Q + g(\lambda)S)$.  The first part of the product is $O_2(1)$ as a map from $L^2(\R^3)$ to $L^{\infty, -1}(\R^3)$ by Lemma~\ref{lem:RvQ}.  The second part of the product, using \eqref{eq:E0 bd}, satisfies the same bounds since $g^+(\lambda)R^+(H_0,\lambda^4) = O_2(\la x\ra \la y\ra)$ and $\la y\ra v \in L^2(\R^3)$.
	
	Finally, the operator $A^+(\lambda)^{-1}vR^+(H_0,\lambda^4)$ is $O_2(1)$ as a map from $L^{1,1}(\R^3)$ to
	$L^2(\R^3)$.  So the composition of these pieces together with $\Gamma_1$ is an $O_2(\lambda)$ map from $L^{1,1}(\R^3)$ to $L^{\infty,-1}(\R^3)$, or in other words it is $O_2(\lambda\la x\ra\la y\ra)$.
\end{proof}

Before proceeding further, we recall from \eqref{eq:R0low} that the expansion of the free resolvent $R^+(H_0,\lambda^4)$ is $\frac{a^+}{\lambda} + G_0 + O_2(\lambda|x-y|^2)$.  Its remainder term is worse than the $O_2(\lambda\la x\ra \la y\ra)$ required for Proposition~\ref{prop:low reg}.  Controlling the first two terms of~\eqref{eq:RV exp reg} requires finding additional cancellation within $\mathcal E_1 = O_2(\lambda|x-y|^2)$.
Using \eqref{eq:4th resolvdef} and \eqref{eq:R0low}, the kernel of $\mathcal{E}_1(\lambda)$ has the form
\be\label{eq:E1 K}
\mathcal{E}_1(\lambda, x, y) = \frac{K(\lambda|x-y|)}{\lambda},\qquad
K(z) = \frac{e^{iz} - e^{-z}}{8\pi z} - a^+ + \frac{z}{8\pi}.
\ee
The functions $K_1(z) = zK'(z)$ and $K_2(z) = zK_1'(z)$ are useful in a moment.  By the Leibniz rule,
\begin{align*}
\frac{d}{d\lambda}\mathcal E_1(\lambda,x,y) &=\frac{-K(\lambda|x-y|) + K_1(\lambda|x-y|)}{\lambda^2}, \\
\frac{d^2}{d\lambda^2}\mathcal E_1(\lambda,x,y) &= \frac{2K(\lambda|x-y|)-3K_1(\lambda|x-y|) + K_2(\lambda|x-y|)}{\lambda^3}.
\end{align*}

All three functions have the property
$|K_j(z)| \les z^2$ and $|K_j'(z)| \les z$ for all positive $z$.  We note that the bound $|K_j(z)|\les z$ is also true.
By the Mean Value Theorem,
\begin{equation} \label{eq:Kjbound1}
|K_j(\lambda|x-y|) - K_j(\lambda |x|)| \les \lambda^2 \la y \ra\max(|x-y|, |x|) \les \lambda^2  \la y \ra\max(\la x\ra, \la y \ra).
\end{equation}  
We may insert $K_j(\lambda|y|)$ into this estimate as below without changing the upper bound, 
$$
\big| K_j(\lambda|x-y|) - K_j(\lambda|x|) - K_j(\lambda|y|)\big| \les \lambda^2  \la y \ra\max(\la x\ra, \la y \ra).
$$  
However the left side is symmetric in $x$ and $y$. thus it is also bounded by $\lambda^2\la x\ra \max(\la x\ra, \la y \ra)$. We may conclude that
\begin{equation} \label{eq:Kjbound2}
\big| K_j(\lambda|x-y|) - K_j(\lambda|x|) - K_j(\lambda|y|)\big| \les \lambda^2\min(\la x\ra, \la y \ra )\max(\la x\ra, \la y \ra) = \lambda^2\la x\ra  \la y \ra.
\end{equation}

\begin{lemma} \label{lem:E_1vQ}
	If $v(x) \les \la x\ra^{-\frac{7}{2}-}$, then $\mathcal E_1vQ$ is $O_2(\lambda)$ as an operator from $L^2(\R^3)$ to
	$L^{\infty, -1}(\R^3)$.
\end{lemma}

\begin{proof}
	For all $\psi \in QL^2$, $\int v(y)\psi(y)\,dy = 0$, hence using \eqref{eq:E1 K},
	\begin{align*}
	|\mathcal E_1v\psi(x)| &= \Big|\int_{\R^3} \lambda^{-1} \big(K(\lambda|x-z|)- K(\lambda|x|)\big)v(z)\psi(z)\,dz\Big| \\
	&\les \lambda \int_{\R^3} \la z\ra \max(\la x\ra, \la z \ra)  v(z)  |\psi(z)|\,dz \les \la x\ra ,
	\end{align*} 
	provided $\la z\ra^2v \in L^2(\R^3)$.  Similar estimates can made for the derivatives of $\mathcal E_1(\lambda)$ using the bounds on
	$K_1$ and $K_2$. 
\end{proof}

For clarity in the calculations to follow, we denote $\mathbf 1$ to be the operator with integral kernel $\mathbf1(x,y)=1$.  That way $R^+(H_0,\lambda^4)$ is $\frac{a^+}{\lambda}\mathbf 1 + \mathcal E_0 = \frac{a^+}{\lambda}\mathbf 1 + G_0 + \mathcal E_1$.  The crucial step in the proof of Proposition~\ref{prop:low reg} is in the estimate below for the first two terms of~\eqref{eq:RV exp reg}.

\begin{lemma}\label{lem:leading terms}
	If $v(x) \les \la x\ra^{-\frac72-}$, then
	\begin{multline*}
	R^+(H_0,\lambda^4) - R^+(H_0,\lambda^4)vA^+(\lambda)^{-1}vR^+(H_0,\lambda^4) \\
	=
	G_0 - \Big[ G_0vQD_0QvG_0 + \frac{c}{\|V\|_1} - \frac{G_0vSv\mathbf{1} +\mathbf{1}vSvG_0}{\|V\|_1}  \Big] + O_2(\lambda\la x\ra\la y\ra)
	\end{multline*}
\end{lemma}

This ultimately explains why the statement of Theorem~\ref{thm:reg} has a lower order of polynomial weights than Theorem~\ref{thm:swave} or the evolution of the unperturbed solution operator.
\begin{proof}
	By \eqref{M inverse regular}, we may write
	\begin{multline}\label{eq:R born}
	R^+(H_0,\lambda^4) - R^+(H_0,\lambda^4)vA^+(\lambda)^{-1}vR^+(H_0,\lambda^4)\\
	=
	R^+(H_0,\lambda^4) -g^+(\lambda)R^+(H_0,\lambda^4)vSvR^+(H_0,\lambda^4)
	-R^+(H_0,\lambda^4)vQD_0QvR^+(H_0,\lambda^4)
	\end{multline}
	The last term can be further expanded as 
	\begin{multline*}
	(G_0+\mathcal E_1(\lambda))vQD_0Qv(G_0+ \mathcal E_1(\lambda))\\ 
	= G_0vQD_0QvG_0 + \mathcal E_1(\lambda)vQD_0Qv  G_0 + G_0vQD_0Qv\mathcal E_1(\lambda)+\mathcal E_1(\lambda)vQD_0Qv\mathcal E_1(\lambda).
	\end{multline*}  
	Here we used the orthogonality of $Q$ to the span of $v$.
	Using Lemma ~\ref{lem:E_1vQ}  as needed, we conclude that
	$$
	R^+(H_0,\lambda^4)vQD_0QvR^+(H_0,\lambda^4) = G_0vQD_0QvG_0 + O_2(\lambda\la x\ra\la y\ra).
	$$
	The function $g^+(\lambda)$ has the power series expansion 
	\be\label{eq:g expansion}
	g^+(\lambda) = \lambda(a^+\|V\|_1 + c\lambda)^{-1} = \frac{\lambda}{a^+\|V\|_1} - \frac{c\lambda^2}{(a^+\|V\|_1)^2} + O_2(\lambda^3),
	\ee
	which allows us to write
	\begin{multline}\label{eq:two terms}
	R^+(H_0,\lambda^4) -g^+(\lambda) R^+(H_0,\lambda^4)vSvR^+(H_0,\lambda^4)\\
	=			R^+(H_0,\lambda^4)-\frac{\lambda}{a^+\|V\|_1} R^+(H_0,\lambda^4)vSvR^+(H_0,\lambda^4)\\
	+\frac{c\lambda^2}{(a^+\|V\|_1)^2} R^+(H_0,\lambda^4)vSvR^+(H_0,\lambda^4)+O_2(\lambda \la x\ra \la y \ra).
	\end{multline}
	The error term follows since Lemma~\ref{lem:RvA} implies that $\lambda R^+(H_0,\lambda^4)=O_2(\la x \ra \la y \ra )$ and $\lambda S = \Gamma_1$. 
	
	Now we take inventory of the order $\lambda^{-1}$, $\lambda^0$, and remainder terms in the rest of \eqref{eq:two terms}.  
	Recalling \eqref{eq:R0low} and Lemma~\ref{regular} we see that
	the only term of order $\lambda^{0}$ 
	arising from $\lambda^2 R^+(H_0,\lambda^4)vSvR^+(H_0,\lambda^4)$ comes from the projection $P$ in the upper-left corner of
	\eqref{def:S}. All other parts of $S$ introduce expressions $R^+(H_0,\lambda^4)vQ$ or $QvR^+(H_0,\lambda^4)$ which neutralizes the leading order part of $R^+(H_0,\lambda^4)$  by Lemma~\ref{lem:RvQ}.  More specifically, writing $\lambda R^+(H_0,\lambda^4)=a^++\lambda \mathcal  E_0(\lambda)$, we have
	\begin{multline}
	\lambda^2 R^+(H_0,\lambda^4)vSvR^+(H_0,\lambda^4)\\ 
	= (a^+)^2\mathbf 1 vSv\mathbf 1 + a^+\lambda( \mathcal E_0(\lambda)vSv\mathbf 1+\mathbf 1 vSv\mathcal E_0 (\lambda) ) + \lambda^2 \mathcal E_0(\lambda) vSv\mathcal E_0(\lambda) \\
	= (a^+)^2\|V\|_1\mathbf 1 + O_2(\lambda\la x\ra \la y\ra). \label{eq:lambda2vSv}
	\end{multline}
	We have used \eqref{def:S} and the fact that $Qv\mathbf 1=\mathbf 1vQ=0$ to see $\mathbf 1vSv\mathbf 1=\mathbf 1vPv\mathbf 1$, the identity $\mathbf 1vPv\mathbf 1 = \|V\|_1 \mathbf 1$, as well as Lemma~\ref{lem:RvQ}.
	
	The next term requires carrying out the expansion a little farther than in \eqref{eq:lambda2vSv} by writing $\mathcal E_0(\lambda)=G_0+\mathcal E_1(\lambda)$ in the expansion of the resolvent.
	\begin{multline*}
	\lambda R^+(H_0,\lambda^4)vSvR^+(H_0,\lambda^4)\\ 
	=	\lambda^{-1}(a^+)^2\mathbf 1 vSv\mathbf 1 + a^+ ( \mathcal E_0(\lambda)vSv\mathbf 1+\mathbf 1 vSv\mathcal E_0 (\lambda) ) + \lambda  \mathcal E_0(\lambda) vSv\mathcal E_0(\lambda) \\
	=  \lambda^{-1}(a^+)^2\|V\|_1\mathbf 1 + a^+(G_0vSv\mathbf 1 + \mathbf 1 vSvG_0) \\
	+ a^+\mathcal E_1(\lambda)vSv\mathbf 1  + a^+\mathbf 1vSv\mathcal E_1(\lambda)+ \lambda\mathcal E_0(\lambda)vSv\mathcal E_0(\lambda) .
	\end{multline*}
	Using \eqref{def:S}, $Qv\mathbf 1=0$, \eqref{eq:E0 bd}, \eqref{eq:E1 bd} and Lemmas \ref{lem:RvQ} and \ref{lem:E_1vQ} as needed, we conclude that
	\begin{multline} \label{eq:RvAvR term}
	\frac{\lambda}{a^+\|V\|_1}R^+(H_0,\lambda^4)vSvR^+(H_0,\lambda^4) = \frac{a^+\mathbf 1}{\lambda}+ \frac{G_0vSv\mathbf 1 + \mathbf 1vSvG_0}{\|V\|_1}  \\ + \frac{\mathcal E_1(\lambda) vPv\mathbf 1 + \mathbf 1 vPv\mathcal E_1(\lambda)}{\|V\|_1}
	+ O_2(\lambda\la x\ra \la y\ra).
	\end{multline}
	We note that by \eqref{eq:E1 bd} the $\mathcal E_1(\lambda)$ terms are order $O_2(\lambda \la x\ra^2 \la y\ra^2)$, which is too large to be included in the remainder.

	The free resolvent has the expansion $R^+(H_0,\lambda^4) = \frac{a^+\mathbf 1}{\lambda} + G_0 + \mathcal E_1(\lambda)$.  The leading-order term cancels immediately with its counterpart in \eqref{eq:RvAvR term} in their contribution to \eqref{eq:R born}.  The lemma, and Proposition~\ref{prop:low reg}, concludes with one last 
	estimate,
	\begin{equation}\label{eq:E1 diffs}
	\mathcal E_1(\lambda) - \frac{\mathcal E_1(\lambda) vPv\mathbf 1 + \mathbf 1vPv\mathcal E_1(\lambda)}{\|V\|_1} = O_2(\lambda\la x\ra \la y\ra).
	\end{equation}
	Using the facts that $Pv\mathbf 1 = v\mathbf 1$, and $\int v^2(z)\, dz=\|V\|_1$, along with \eqref{eq:E1 K}, the operator on the left side of \eqref{eq:E1 diffs} has kernel
	\begin{multline} \label{eqn:weird-cancellation}
	\frac{1}{\lambda}\bigg(K(\lambda|x-y|) - \int K(\lambda|x-z|)\frac{v^2(z)}{\|V\|_1}\,dz - \int K(\lambda|y-z|)\frac{v^2(z)}{\|V\|_1}\,dz \bigg) \\
	= \frac{K(\lambda|x-y|) - K(\lambda|x|) - K(\lambda|y|)}{\lambda} 
	- \int\frac{K(\lambda|x-z|) - K(\lambda|x|)}{\lambda}\frac{v^2(z)}{\|V\|_1}\,dz \\
	- \int \frac{K(\lambda|y-z|) - K(\lambda|y|)}{\lambda}\frac{v^2(z)}{\|V\|_1}\,dz.
	\end{multline}
	From the inequalities \eqref{eq:Kjbound1}, \eqref{eq:Kjbound2} for $K$ we can see that
	$$
	\Big|\mathcal{E}_1(\lambda) - \mathcal{E}_1(\lambda)\frac{v^2}{\|V\|_1}\mathbf{1} - \mathbf{1}\frac{v^2}{\|V\|_1}\mathcal{E}_1(\lambda) \Big| 
	\les \lambda\la x\ra \la y\ra  + \lambda \int (\la x\ra+\la y\ra) \frac{\la z\ra^2v^2(z)}{\|V\|_1}\,dz \les \lambda\la x\ra \la y\ra
	$$
	provided $|v(x)| \les \la x\ra^{-\frac52-}$.  
	Bounds for the first two derivatives can be computed from \eqref{eqn:weird-cancellation} by using the Leibniz rule formulas for $\frac{d}{d\lambda}\mathcal E_1(\lambda)$ and $\frac{d^2}{d\lambda^2}\mathcal E_1(\lambda)$ and the analogous properties of $K_1$ and $K_2$.
\end{proof}

Finally, we prove the main Proposition.

\begin{proof}[of Proposition~\ref{prop:low reg}]
	Inserting the results of Lemma~\ref{lem:RvA}, \ref{lem:leading terms}
	into \eqref{resid}, specifically the expansion \eqref{eq:RV exp reg} suffices to prove the desired result.

\end{proof}

We collect a couple of notes for future reference when handling the resonant case in the next section.  First, the cancellation of
order $\lambda^{-1}$ terms between $R^+(H_0,\lambda^4)$ and $R^+(H_0,\lambda^4)vg^+(\lambda)SvR^+(H_0,\lambda^4)$  takes place in the same manner because the projection in the upper left corner of $S$ in \eqref{def:S} is the same in either case.

The fact that the constant, order $\lambda^0$ terms, are the same for $R_V^+(\lambda^4)$ and $R_V^-(\lambda^4)$ in the resonant case is much more difficult to check by hand. However, note that the original definition of $R_V^\pm(\lambda^4) = R_V(\lambda^4 \pm i0)$
implies that $R_V^-(\lambda^4) = R_V^+((i\lambda)^4)$.  As a consequence, the power series coefficients for
$R_V^\pm(\lambda^4)$ should coincide for each term of order $\lambda^{4k}$ and for the constant term in particular.

\section{Weighted bounds when there is a resonance of the first kind}\label{sec:swave}

The main goal of this section is to establish the following 

\begin{prop}\label{prop:low res}
	If there is a resonance of the first kind at zero and $|V(x)|\les \la x\ra^{-\beta}$ for some $\beta>11$, then we have the following expansion for the perturbed resolvent in a sufficiently small neighborhood of $\lambda=0$:
	\be\label{eq:RV res}
	R_V^\pm(\lambda^4)(x,y) = \frac{a^\pm \|V\|_1}{\lambda} C_{-1}^\pm+C_0+ O_2(\lambda \la x\ra ^2\la y \ra^2).
	\ee
	Where $C_{-1}^{\pm}$ is an operator of rank at most four given by
	\begin{align*}
	C_{-1}^\pm =(G_0v + a^\pm\mathbf 1vF^\pm_L)D_1(vG_0 + a^\pm F^\pm_Rv\mathbf 1)
	%	C_{-1}^\pm =G_0vD_1vG_0
	%	+(a^\pm)  G_0vD_1F_R^\pm v  \mathbf{1}
	%	+(a^\pm) \mathbf{1}vF_L^\pm D_1vG_0
	%	+ (a^\pm)^2 \mathbf{1} vF_L^\pm D_1 F_R^\pm v   \mathbf{1}
	\end{align*}
	where
	$F_R^\pm:= \frac{1}{a^\pm \|V\|_1}S+a_1^\pm  vG_1vQD_0Q ,$
	and $F_L^\pm:= \frac{1}{a^\pm \|V\|_1}S+a_1^\pm  QD_0QvG_1v $.
	In particular,
	\be\label{eq:Rv diff res}
	R_V^+-R_V^- = \frac{C_{-1}^+-C_{-1}^- }{\lambda}  + O_2(\lambda \la x\ra ^2\la y \ra^2).
	\ee
	
\end{prop}

From this theorem we can write an explicit formulation of the time dependent operator $F_t$ in the statement of Theorem~\ref{thm:swave},
\be \label{Ft defn}
F_t=\frac{2}{\pi i} \int_0^\infty e^{-it\lambda^4} \lambda^3 \chi(\lambda) \frac{C_{-1}^+-C_{-1}^- }{\lambda}\, d\lambda
\ee
Since $D_1=S_1D_1S_1$, we can see that $C^\pm_{-1}$ is an operator of rank at most the rank of $S_1$.  By Remark~7.9 in \cite{egt}, in the case of a resonance of the first kind the rank of $S_1$ is at most four.
As in the case when zero is regular, we need to understand the resolvent as $\lambda\to 0$, hence we study the operators $M^\pm(\lambda)^{-1}$.

\begin{lemma}\label{lem:Minvexp} 
	If there is a resonance of the first kind at zero and if $v(x)\les \la x\ra^{-\frac{11}2-}$, then
	$$[M^\pm(\lambda) ]^{-1} =(M^\pm(\lambda)+S_1)^{-1}
	-\frac{a^\pm \|V\|_1 D_1}{\lambda}+\widetilde M_0^\pm +\lambda \widetilde M_1^\pm
	+\lambda^2 \widetilde M_2^\pm+\Gamma_{3}.$$
	Furthermore, $D_1$, $\widetilde M_0^\pm $, $\widetilde M_1^\pm$ and $\widetilde M_2^\pm$ are finite-rank operators.
\end{lemma} 

Our expansions vary from those presented in \cite{egt} as we need further information on the second derivative for our goals.  We begin with the following lemma from \cite{JN}.
\begin{lemma}\cite[Lemma~2.1]{JN}\label{JNlemma}
	Let $M$ be a closed operator on a Hilbert space $\mathcal{H}$ and $S$ a projection. Suppose $M+S$ has a bounded
	inverse. Then $M$ has a bounded inverse if and only if
	$$
	B:=S-S(M+S)^{-1}S
	$$
	has a bounded inverse in $S\mathcal{H}$, and in this case
	\begin{align} \label{Minversegeneral}
	M^{-1}=(M+S)^{-1}+(M+S)^{-1}SB^{-1}S(M+S)^{-1}.
	\end{align}
\end{lemma}
We use this lemma with $M=M^{\pm}(\lambda)$ and $S=S_1$ to prove Lemma~\ref{lem:Minvexp}.  Then, we have $B^\pm(\lambda)=S_1-S_1(M^{\pm}(\lambda)+S_1)^{-1}S_1$ on $S_1L^2$.  We note that since $S_1$ is a finite rank projection, both $B^\pm(\lambda)$ and $B^\pm(\lambda)^{-1}$ are finite rank operators.  In the case of a resonance of the first kind, $S_1$ has rank at most four.

\begin{proof}[of Lemma~\ref{lem:Minvexp}]

	Using \eqref{M_0} and \eqref{M inverse regular}, via a longer Neumann series expansion than in \eqref{eq:Minv exp reg} when $S_1=0$, we obtain  
	\begin{multline}\label{M+S1ex2}
	(M^\pm (\lambda)+S_1)^{-1} = QD_0Q+g^\pm(\lambda)S -a_1^\pm\lambda QD_0QvG_1v QD_0Q \\
	-a_1^\pm\lambda g^\pm(\lambda)\big[QD_0QvG_1v S+SvG_1v  QD_0Q\big]+ (a_1^\pm)^2\lambda^2 QD_0Q (vG_1v  QD_0Q)^2 +\Gamma_{3},
	\end{multline}
	provided that $v(x)\les \la x\ra^{-\frac{11}2-}$.  Here we note that all the operators are $\lambda$ independent, and we carefully note any $\pm$ dependence.
	\begin{multline}\label{S1M+S1S1}
	S_1(M^\pm (\lambda)+S_1)^{-1}S_1 = S_1+ g^\pm (\lambda) T_1  - a_1^\pm \lambda g^\pm (\lambda) S_1vG_1vS_1 \\
	-a_1^\pm\lambda g^\pm  (\lambda)S_1\big[vG_1v S+SvG_1v \big]S_1+ (a_1^+)^2\lambda^2 S_1 vG_1v QD_0QvG_1v S_1  +\Gamma_{3}.
	\end{multline}
	Therefore 
	\begin{multline*}
	B^\pm (\lambda)= S_1- S_1(M ^\pm (\lambda)+S_1)^{-1}S_1 = -g^\pm (\lambda)T_1\\+  a_1^\pm\lambda g^\pm (\lambda) 
	S_1( vG_1v  + SvG_1v   +vG_1v S )S_1 
	-(a_1^\pm)^2 \lambda^2  S_1vG_1v QD_0QvG_1v S_1+S_1\Gamma_{3}S_1.
	\end{multline*}
	Recalling \eqref{eq:g expansion},
	a careful Neumann series computation yield the expansion
	\be \label{eq:Binv pm}
	B^{\pm}(\lambda)^{-1} = - \frac{a^\pm\|V\|_1}{\lambda} D_1 - c D_1 - a_1^\pm a^\pm D_1\tilde{B}_0 D_1 + (a_1^\pm a^\pm)^2   D_1 \tilde{B}_{0,1} D_1 + D_1\Gamma_1D_1.
	\ee
	Here the operators $\tilde {B}_0$ and $ \tilde{B}_{0,1}$ are independent of both $\lambda$ and the choice of $\pm$.
	Substituting \eqref{eq:Binv pm} and \eqref{M+S1ex2} into \eqref{Minversegeneral} finishes the proof.

\end{proof}

We are now ready to prove Proposition~\ref{prop:low res}.  

\begin{proof}[of Proposition~\ref{prop:low res}]
	By the symmetric resolvent identity, we need to understand the behavior of 
	\begin{multline*}
	R_V^\pm(\lambda^4)=R^\pm(H_0,\lambda^4) -R^\pm(H_0,\lambda^4) (M^\pm(\lambda)+S_1)^{-1} R^\pm(H_0,\lambda^4)\\ - R^\pm(H_0,\lambda^4)(M^\pm(\lambda)+S_1)^{-1} S_1 B^\pm(\lambda)^{-1}S_1 (M^\pm(\lambda)+S_1)^{-1} R^\pm(H_0,\lambda^4). 
	\end{multline*}
	The contribution of the first two terms, and the cancellation of their most singular $\lambda$ terms follow exactly as in the case when there is no resonance at zero, that is when $S_1=0$.  The expansion in Proposition~\ref{prop:low reg} implies that  
	$$
	R^\pm(H_0,\lambda^4) -R^\pm(H_0,\lambda^4)(M^+(\lambda)+S_1)^{-1} R^\pm(H_0,\lambda^4)=C_{0,1}+O_2(\lambda \la x\ra \la y \ra  ).
	$$  
	This follows since the leading terms of the expansion in $(M^\pm(\lambda)+S_1)^{-1}$ and $M^\pm(\lambda)^{-1}$, on which the delicate cancellation relies, are identical. Namely, by \eqref{S1M+S1S1}
	$$
	(M^\pm (\lambda)+S_1)^{-1} = QD_0Q+g^\pm(\lambda)S+Q\Gamma_1Q.
	$$
	The remaining terms in \eqref{M+S1ex2} is controlled using  that Lemma~\ref{lem:RvQ} implies that
	$$
	R^\pm (H_0,\lambda^4)vQ\Gamma_1QvR^\pm (H_0,\lambda^4)
	=O_2(\lambda \la x\ra \la y \ra ).
	$$
	For the remainder of the proof, we seek to establish that the order zero terms in \eqref{eq:RV res} are independent of the choice of `+' or `-'.	 For simplicity, we drop the superscripts and consider the `+' case unless otherwise noted.
	We write
	$$
	(M+S_1)^{-1}=QD_0Q+\lambda M_1+\lambda^2 M_2+\Gamma_3.
	$$
	Here  $M_1=\frac{1}{a\|V\|_1}S+a_1QD_0QvG_1vQD_0Q$ and a careful Neumann series computation yields
	\begin{multline}\label{eq:M2 pm}
	M_2=\frac{-c}{(a\|V\|_1)^2}S-\frac{a_1}{a\|V\|_1} QD_0QvG_1vS-\frac{a_1}{a\|V\|_1} SvG_1vQD_0Q\\
	+ a_1^2 QD_0QvG_1vQD_0QvG_1vQD_0Q .
	\end{multline}
	Since $B=[g(\lambda) T_1+O_2(\lambda^2)]$, and
	using \eqref{eq:g expansion}, as in \eqref{eq:Binv pm}
	we may write
	$$
	B^{-1}=\frac{1}{\lambda}B_{-1}+B_0+\lambda B_1,
	$$
	here $B_{-1}=a\|V\|_1D_1$ where $D_1=S_1T_{1}^{-1}S_1$ the operators are explicit in \eqref{eq:Binv pm}.
	
	Then, we have the following as expansion for sufficiently small $\lambda$:
	\begin{multline*}
	(M+S_1)^{-1}S_1B^{-1}S_1 (M+S_1)^{-1}\\
	=\bigg[ QD_0Q+\lambda M_1+\lambda^2 M_2 \bigg] \bigg[ \lambda^{-1}B_{-1}+B_0+\lambda B_1 \bigg]\bigg[ QD_0Q+\lambda M_1+\lambda^2 M_2 \bigg]\\
	=\lambda^{-1}QD_0QB_{-1}QD_0Q+\bigg[ QD_0QB_{-1}M_1+M_1B_{-1}QD_0Q+QD_0QB_0QD_0Q \bigg]\\
	+\lambda \bigg[ M_2B_{-1}QD_0Q+QD_0QB_{-1}M_2	+M_1B_0QD_0Q+QD_0QB_0M_1\\
	+M_1B_{-1}M_1+QD_0QB_1QD_0Q \bigg]\\
	+\lambda^2 \bigg[ M_1B_0M_1+M_1B_{-1}M_2+M_2 B_{-1}M_1+QD_0Q\Gamma_0+\Gamma_0 QD_0Q \bigg]+ \Gamma_{3}.
	\end{multline*}
	Using
	\[
	R (H_0, \lambda^4) = \frac{a  \mathbf{1}}{\lambda} + G_0 + a_1  \lambda G_1 + \mathcal E_2(\lambda),
	\]
	and that $\mathbf 1 v Q=Qv\mathbf 1=0$, we arrive at the following expansion
	\begin{multline}\label{eq:ugly expansion}
	R (H_0,\lambda^4)v(M+S_1)^{-1}S_1B^{-1}S_1 (M+S_1)^{-1}vR (H_0,\lambda^4)=\\ 
	\frac{1}{\lambda} \bigg[G_0vQD_0QB_{-1}QD_0QvG_0+ G_0vQD_0QB_{-1}M_1va\mathbf{1}\\
	+a\mathbf{1}vM_1B_{-1}QD_0QvG_0+ a\mathbf{1} vM_1B_{-1}M_1v a\mathbf{1} \bigg]\\
	+\bigg[ a_1 G_1vQD_0QB_{-1}QD_0QvG_0+a_1 G_0vQD_0QB_{-1}QD_0QvG_1
	+G_0vQD_0QB_{-1}M_1vG_0 \\
	+G_0vM_1B_{-1}QD_0QvG_0 +a_1G_1vQD_0QB_{-1}M_1va\mathbf{1}+a\mathbf{1}vM_1B_{-1}QD_0QvG_1a_1\\
	+G_0vQD_0QB_0QD_0QvG_0 
	+a\mathbf{1}v M_2B_{-1}QD_0QvG_0+G_0vQD_0QB_{-1}M_2va\mathbf{1}\\
	+a\mathbf{1}vM_1B_0QD_0QvG_0+G_0vQD_0QB_0M_1va\mathbf{1}
	+G_0v M_1B_{-1}M_1v a\mathbf{1}+a\mathbf{1} vM_1B_{-1}M_1vG_0\\
	+a\mathbf {1}vM_1B_0M_1va \mathbf{1}+a\mathbf {1}vM_1B_{-1}M_2va \mathbf{1}+a\mathbf {1}vM_2 B_{-1}M_1va\mathbf{1} 
	\bigg] 
	+O_2(\lambda \la x\ra^2 \la y\ra^2).
	\end{multline}
	Our goal is to show that the order $\lambda^0$ terms are all independent of the `+' or `-' in the resolvent.  Recalling \eqref{adef}, we see $a^+a_1^+=a^-a_1^-$, which allows us to conclude the following pairs from the above expansion are $\pm$ independent: $a_1B_{-1}$, $M_1B_{-1}$.  This also implies that $(aa_1)^2$, and hence the combinations $a_1B_{-1}M_1a$, and $aM_1$ are $\pm$ independent.
	Further, it follows from \eqref{eq:Binv pm} that $B_0$
	is $\pm$ independent.  We may also conclude that the combination $aM_1B_0$ is $\pm$ independent.
	Finally, using \eqref{eq:M2 pm} we may conclude that the combinations $a M_2 B_{-1}$ and $aM_2B_{-1}M_1a$
	are $\pm$ independent.
	
	These computations suffice to establish the claim that the order $\lambda^0$ terms cancel in the $+/-$ difference.  Using that $S_1QD_0Q=QD_0QS_1=S_1$ on the order $\lambda^{-1}$ operator in \eqref{eq:ugly expansion} yields the formula for $C_{-1}^\pm$ given in Proposition~\ref{prop:low res}.

\end{proof}

Proposition~\ref{prop:low res}, along with oscillatory integral bounds in Lemmas~\ref{lem:t+a4bound} and \ref{lem:t-34bound} yields the dispersive bound
$$
e^{-itH}P_{ac}(H)\chi(H)=A_{-1}\la t\ra^{-\f34}+ A_1 \la t\ra^{-\f54},
$$
where $A_{-1}$ is an operator from $L^1\to L^\infty$ and $A_1:L^{1,2}\to L^{\infty, -2}$.

\begin{rmk}\label{rmk:wtd}
	We note that the error term in Lemma~\ref{lem:Minvexp}, specifically \eqref{M+S1ex2} may be replaced by $\Gamma_{2+}$, which requires only that $v(x)\les \la x\ra^{-\f92-}$.  This filters through the argument, allowing one to replace the error terms $\Gamma_k$ with $\Gamma_{k-1+}$, and the error term in \eqref{eq:ugly expansion} is now $O_2(\lambda^{0+} \la x\ra^2\la y \ra ^2)$.
	This yields the expansion
	$$
	e^{-itH}P_{ac}(H)\chi(H)=A_{-1}\la t\ra^{-\f34}+ A_1 \la t\ra^{-1-},
	$$
	where $A_{-1}$ is an operator from $L^1\to L^\infty$ and $A_1:L^{1,2}\to L^{\infty, -2}$ requiring only $|V(x)|\les \la x\ra^{-\beta}$ for $\beta>9$.
	
\end{rmk}

\section{High Energy}\label{sec:high}

To complete the proofs of Theorem~\ref{thm:reg} and \ref{thm:swave}, we need to control the contribution of the high energy portion, $\lambda \gtrsim 1$, in \eqref{stone}.  The effect of a zero energy resonance is only felt in a neighborhood of zero.  Hence, we need only prove one bound for the high energy.

\begin{prop} \label{prop:large}Let $|V(x) | \les \la x \ra^{-5-}$, and assume there are no embedded eigenvalues in the continuous spectrum of $H$, then  for $|t|>1$
	\begin{align} \label{eq:large}
	\Bigg|\int_0^\infty e^{-it\lambda^4} \lambda^3 \widetilde{\chi}(\lambda) R_V^{\pm} (\lambda^4)  (x,y) \, d\lambda \Bigg|\les | t |^{-2}\la x\ra \la y \ra.
	\end{align}
\end{prop}
Recall that $\widetilde\chi(\lambda)=1-\chi(\lambda)$ is a cut-off away from a small neighborhood of $\lambda=0$.
In Proposition~6.1 of \cite{egt}, the uniform bound:
\begin{align} \label{eq:large unwtd}
\sup_{x,y\in \mathbb R^3}\Bigg|\int_0^\infty e^{-it\lambda^4} \lambda^3 \widetilde{\chi}(\lambda) R_V^{\pm} (\lambda^4)  (x,y) \, d\lambda \Bigg|\les | t |^{-\frac34}
\end{align}
was established.  This bound holds true under the hypotheses of Proposition~\ref{prop:large}.  The bound in \eqref{eq:large} decays faster for large $t$ at the cost of spatial weights.  This bound can be established directly by integration by parts. Using the support condition of $\widetilde \chi(\lambda)$ and the decay of $R_V^\pm(\lambda^4)$ as $\lambda\to \infty$ established below, there will be no boundary terms and it suffices to control
\be \label{eq:hi ibp}
\Bigg|\int_0^\infty e^{-it\lambda^4} \lambda^3 \widetilde{\chi}(\lambda) R_V^{\pm} (\lambda^4)  (x,y) \, d\lambda \Bigg|\les \frac{1}{t^2}\int_0^\infty \Bigg| \partial_{\lambda} \bigg(  \frac{\partial_{\lambda} [\widetilde{\chi}(\lambda) R_V^{\pm} (\lambda^4)  (x,y) ]}{\lambda^3}\bigg)\Bigg|\, d\lambda.
\ee
To prove the Proposition~\ref{prop:large} we use the resolvent identities and write, 
\begin{align}\label{large symmetric}
R_V(\lambda^4)= R^\pm(H_0, \lambda^4)  - R^\pm(H_0, \lambda^4) VR^\pm (H_0, \lambda^4) + R^\pm(H_0, \lambda^4) VR_V(\lambda^4) V R^\pm (H_0, \lambda^4).
\end{align} 
Recall by \eqref{eq:R0low}, we have 
\begin{align} \label{est1}
R^\pm(H_0, \lambda^4)  (x,y)  =O_1(\lambda^{-1}), \qquad
\partial_{\lambda}^2R^\pm(H_0, \lambda^4)  (x,y)  =O(\lambda^{-2}\la x\ra \la y \ra).
\end{align}
This, along with the fact that $\lambda\gtrsim 1$, shows that 
\begin{align}\label{est2}
\partial_{\lambda}^k[R^\pm (H_0, \lambda^4)VR^\pm(H_0, \lambda^4)]=O(\lambda^{-2}\la x\ra \la y \ra), \qquad k=0,1,2.
\end{align}
This suffices to establish the desired bound for the contribution of the first summand of \eqref{large symmetric} to \eqref{eq:hi ibp}.  In fact, the spatial weights are only necessary when $k=2$.  Furthermore, these bounds suffice to control the second summand of \eqref{large symmetric} to \eqref{eq:hi ibp} as they imply (for $k=0,1,2$)
\begin{align*}
&|\partial_{\lambda}^k[R^\pm (H_0, \lambda^4)VR^\pm(H_0, \lambda^4)(x,y)]| \les \lambda^{-2}  \la x\ra \la y\ra\int_{\R^3} \la x_1\ra |V (x_1)| dx_1\les \lambda^{-2}  \la x\ra \la y\ra.
\end{align*}
The assumed decay on $V$ suffices to ensure the integral converges.

For the final summand in \eqref{large symmetric}, we utilize the limiting absorption principle.
\begin{theorem} \label{th:LAP}\cite[Theorem~2.23]{fsy} Let $|V(x)|\les \la x \ra ^{-k-1}$. Then for any $\sigma>k+1/2$, $\partial_z^k R_V(z) \in \mathcal{B}(L^{2,\sigma}(\R^d), L^{2,-\sigma}(\R^d))$ is continuous for $z \notin {0}\cup \Sigma$. Further, 
	\begin{align*} 
	\|\partial_z^k  R_V(z) \|_{L^{2,\sigma}(\R^d) \rar L^{2,-\sigma}(\R^d)} \les z^{-(3+3k)/4}. 
	\end{align*}
\end{theorem}

Note that first by \eqref{est1}, and using that $L^\infty \subset L^{2,-\f32-}$, we have 
\begin{align}
\| [R^\pm (H_0, \lambda^4)] \|_{ L^1 \rightarrow L^{2,-\sigma} }= O_1(\lambda^{-1}), \,\,\,\ \sigma>3/2,
\end{align}
which implies the dual estimate as an operator from $L^{2,\sigma}\to L^\infty$.  
Hence, by Theorem~\ref{th:LAP} we have the following estimate
\begin{align*}
& | [R^\pm (H_0, \lambda^4) VR_V(\lambda^4) V R^\pm (H_0, \lambda^4)] | \\ 
&\les \|R^\pm (H_0, \lambda^4) \|_{ L^{2,\sigma} \rightarrow L^{\infty}} \| V R_V(\lambda^4) V \|_{ L^{2,- \sigma} \rightarrow L^{2,\sigma} } \| R^\pm (H_0, \lambda^4) \|_{ L^ {1} \rightarrow L^{2,-\sigma} } \les \lambda^{-2},
\end{align*}
provided $|V(x) | \les \la x \ra^{-2-}$.
Similarly, by \eqref{est1} and Theorem~\ref{th:LAP} (with $z=\lambda^4$) one obtains the same bound for $\partial_{\lambda} [R^\pm (H_0, \lambda^4) VR_V(\lambda^4) V R^\pm (H_0, \lambda^4)]$.  Finally,
\begin{align}\label{est3}
| \partial_{\lambda}^2 \{R^\pm(H_0, \lambda^4) VR_V(\lambda^4) V R^\pm(H_0, \lambda^4) \}| \les \lambda^{-2}\la x\ra \la y\ra 
\end{align}
provided $|V(x) | \les \la x \ra^{-5-}$. Here, we note that the extra decay on $V$ is needed when the derivatives fall on the perturbed resolvent $R_V$ so that $V$ maps $L^{2,-\f52-}\to L^{2,\f52+}$.  The weights arise from when both derivatives fall on a single free resolvent, \eqref{est1}.
%	then \eqref{est1} implies
%	\begin{align*}
%	\| \la \cdot \ra^{-1} \partial_{\lambda}^2[R^\pm (H_0, \lambda^4)] \|_{ L^{1} \rightarrow L^{2,-\sigma} }= O(\lambda^{-2}), \,\,\,\ \sigma>5/2,
%	\end{align*}

\begin{proof}[of Proposition~\ref{prop:large}]
	Substituting \eqref{est1}, \eqref{est2} and \eqref{est3}
	into \eqref{large symmetric}
	we have shown that
	$$
	|\partial_\lambda ^k R_V(\lambda^4)| \les \lambda^{-2} \la x \ra  \la y \ra, \qquad k=0,1,2 .
	$$
	Hence, by \eqref{eq:hi ibp}, we have 
	$$
	\Bigg|\int_0^\infty e^{-it\lambda^4} \lambda^3 \widetilde{\chi}(\lambda) R_V^{\pm} (\lambda^4)  (x,y) \, d\lambda \Bigg|\les \frac{\la x \ra  \la y \ra  }{t^2}\int_1^\infty \lambda^{-5} \, d\lambda\les \frac{\la x \ra  \la y \ra  }{t^2}.
	$$
	
\end{proof}	

We are now ready to prove the main theorems.

\begin{proof}[of Theorems~\ref{thm:reg} and \ref{thm:swave}]
	Interpolating between the high energy bounds \eqref{eq:large} and \eqref{eq:large unwtd} suffice to show that the large energy portion of the solution satsifies the required bounds in either case.
	
	For the low energy portion when zero is regular, we note that Proposition~\ref{prop:low reg}, specifically \eqref{eq:reg Rv diff} may be inserted into the Stone's formula \eqref{stone}.  Lemma~\ref{lem:t+a4bound} provides the desired time decay.  When there is a resonance of the first kind at zero, we use \eqref{eq:Rv diff res} in Proposition~\ref{prop:low res} in the Stone's formula.  Now, the operator $F_t$ is controlled by Lemma~\ref{lem:t-34bound} and the remaining piece is controlled by Lemma~\ref{lem:t+a4bound}. 
	
\end{proof}

\section*{Acknowledgments}
The authors wish to thank the reviewers whose detailed and careful reports greatly improved the paper.

\end{document}